\documentclass{amsart}
\usepackage{}
\usepackage{amsmath}
\usepackage{paralist}
\usepackage{amssymb}
\usepackage{amsthm}
\usepackage{amscd}

\allowdisplaybreaks[4]
\usepackage{graphicx,mathrsfs}
\usepackage[colorlinks=true]{hyperref}
\hypersetup{urlcolor=blue, citecolor=red}

\numberwithin{equation}{section}
\newtheorem{Thm}{Theorem}[section]
\newtheorem{Lem}{Lemma}[section]

\newtheorem{Rem}{Remark}[section]

\begin{document}
\title{Stackelberg games with the third party}
\thanks{Acknowledgements: This work is supported by the NSFC under the grands 12271269 and supported by the Fundamental Research Funds for the Central Universities.}
\author{Yiming Jiang}
\address{School of Mathematical Sciences and LPMC\\ Nankai University\\ Tianjin 300071 China}
\email{ymjiangnk@nankai.edu.cn}
\author{Yawei Wei}
\address{School of Mathematical Sciences and LPMC\\ Nankai University\\ Tianjin 300071 China}
\email{weiyawei@nankai.edu.cn}
\author{Jie Xue}
\address{School of Mathematical Sciences \\ Nankai University\\ Tianjin 300071 China}
\email{1120200032@mail.nankai.edu.cn}
	
	
	
\begin{abstract}
In this paper, we introduce the third party to achieve the Stackelberg equilibrium with the time inconsistency in three different
Stackelberg games, which are the discrete-time games, the dynamic games, and the mean field games. Here all followers are experiencing learning-by-doing. The role of a third party is similar to industry associations, they supervise the leader's implementation
and impose penalties for the defection with the discount factor. Then we obtain different forms of discount factors in different models
and effective conditions to prevent defection.
These results are consistent and the third party intervention is effective and maneuverable in practice.
\end{abstract}

\maketitle	

\section{INTRODUCTION}
\subsection{Research history and motivations}
The model of learning-by-doing has often been identified as one of the important sources of economic growth. During the production process, knowledge as the by-product may improve the technique level of the producer, which is clearly different from that in the research institutions. The earliest work about learning-by-doing model can be traced back to Arrow \cite{55} in 1962. Later Sheshinski \cite{56} and Romer \cite{57} introduced this idea in the neoclassical growth model with optimizing agents. 
 Benchekroun et al. \cite{12} discussed a duopoly market in one-shot Stackelberg games. 

Stackelberg games are hierarchical ones where there are several leaders with a dominant position over the followers, which was first proposed by Stackelberg in \cite{27}. There are many scholars studied the Stackelberg game, see Bensoussan et al. \cite{51} and Moon \cite{52}.
For mean field Stackelberg games, the mean field behavior of the followers will depend on an arbitrary strategy of the leader. Mean field games (briefly, MFG) theory is devoted to the analysis of differential games with infinitely many players. This theory has been introduced by Lasry and Lions in 2006 \cite{1} \cite{2}. At about the same time, Huang et al. \cite{3} solved the large population games independently. 
The MFG problem with major and minor players was considered in Moon and Bassar \cite{9}, 
 Nguyen and Huang \cite{32}. The mass behavior of the minor players is strongly affected by the major players through the cost function and stochastic differential equations (briefly, SDE). Here SDE is driven by the white noise, but not limited to white noises, see Bo and Yang \cite{102}.

The time inconsistency is an obviously important problem for Stackelberg games, which makes Stackelberg equilibrium inappropriate.
Dating back to Strotz \cite{30}, there are three approaches to handle the time inconsistency: precommitment, equilibrium and naive strategies.
 A precommitment strategy optimizes the objective functional anticipated at the very beginning time and the leader sticks with this strategy over the whole game period, see Zhou and Li \cite{2000}. An equilibrium strategy consistently optimizes the objective functional  anticipated at every time point in the similar manner of dynamic programming but using the concept of subgame perfect equilibrium, see Pun \cite{Pun}. A naive strategy is such that, at any time, it coincides with the optimal strategy with respect to the preference at that time. In fact, the theory asserts that, smart followers will not believe any promises made by the leader, and therefore the leader cannot adopt the naive strategy in this way.

 However, there are some disadvantages to precommitment and equilibrium strategies. On the one hand, the precommitment strategy is sometimes
 inferior because its strong commitment leads to time inconsistency in efficiency, see Cui et al. \cite{2010}, and its error-accumulation
 property brings huge estimation errors, see Chiu et al. \cite{2017}. On the other hand, since precommitment is not always a feasible solution to the problem  of intertemporal conflict, the leader with insight into his future unreliability may adopt an equilibrium strategy and reject any plan which  he will not follow through. The leader's goal is then to find the best plan among those that he will actually follow. Although the equilibrium strategy  is time consistent, it emphasizes local optimality which is always not globally optimal.

 Compared with precommitment and equilibrium strategies, the Stackelberg equilibrium strategy sometimes enables to gain more benefits for all players.
 Due to the time inconsistency of Stackelberg equilibrium, followers find that the risk of trusting the leader's commitment is significant,
 and the leader also finds it difficult to withstand the temptation in the future. Introducing a third party is a good approach for leaders and followers in the Stackelberg game, even if players have to pay a certain fee. This idea is inspired by the trigger strategy for Nash games in  Dockner et al. \cite{11}. In \cite{11}, to keep the Nash equilibrium, every player has to agree to follow a certain target path and sustain this agreement with an effective threat to regulate any defector.

This motivated us to introduce the third party in the Stackelberg games with time inconsistency. Here the duty of the third party is to enable all players to keep the Stackelberg equilibrium strategy. In this paper, the Stackelberg game with the third party is set as follows. First,  the followers need to pay  part of the fee and the leader needs to pay a certain deposit to begin with the game. Second, the third party will supervise the leader's execution during the game, and impose mandatory penalties in the event of his defection. The appreciated penalties by the third party can ensure Stackelberg equilibrium is credible. Thus it provides a theoretical basis for the game rules. Third, the leader's deposit is returned and the followers pay the remaining fee at the end of the game. Note that the Stackelberg equilibrium with the third party can be regarded as a modified naive strategy. This discussion is similar to an inverse game problem, except that the third party here is not all players. A similar role to the third party is an industry association, government department, or banking institution.

\subsection{Statement of the problems and main results}
In this subsection, we first set up three different Stackelberg games, which are the discrete-time games, the dynamic games and the MFG. In the Stackelberg games, the leader announces his optimum strategy in advance by taking into account the rational reactions of the followers, where the followers are experiencing learning-by-doing with fewer knowledge. Then followers choose their optimum strategies non-cooperatively based on the leader's strategy. Finally, the leader comes back and implements his announced strategy. Hence we obtain the resulting optimum strategies for the leader and followers by the Pontryagin maximum principle, which forms the Stackelberg equilibrium. In fact, as a rational player, the leader will deviate from his commitment without the followers' knowledge and we obtain the defection strategy. Hence the Stackelberg equilibrium is time inconsistent.
 Next, we introduce the third party with appreciate mandatory penalties for the leader's payoff in the games. Finally, we obtain the corresponding condition to effectively prevent defection, which makes the Stackelberg equilibrium be credible.

{\bf{Model 1: Discrete-time Stackelberg games model.}} This model is a repeated oligopoly Stackelberg game, in which the same game is played in each period. The repeated game here consists of the following static game as the original game, which comes from \cite{12}.
 We consider a duopoly market consisting of firm-0 and firm-1 competing as Stackelberg rivals, which produce a homogeneous good.
  Two firms repeat the duopoly games $N$ times and the period is $[0,T]$. They both move one shot                                                                                  at the beginning of each period. Denote $u_{i}(t)$ as
   the output of the firm-$i$, $i\in\{0,1\}$ at time t, $x_{i}(t)$ as his stock of knowledge, and $c_{i}$ as his unit cost.
   Assume that the stock $x_i\in[0,1]$, and the inverse demand function $K$ satisfies
$K=a-b(u_{1}+u_{0}),$  where $a$ and $b$ are positive constants. As the leader, firm-0 has complete knowledge and  dominates the other. Hence we assume that
\begin{equation}\label{x0}
\dot{x}_0(t)=0,
\end{equation}
with initial stocks $x_0(0)$, and $c_0\leq c_1$. It implies that his stock of knowledge is invariant, and his unit cost is the lowest than others. For the followers, firm-1 is experiencing learning-by-doing and his knowledge subjects to
\begin{equation}\label{9.2}
\dot{x}_{1}(t)=u_{1}(t)-\delta x_{1}(t),
\end{equation}
with the initial condition $ x_{1}(0)=0$, where $\delta>0$ is the rate of depreciation of knowledge.
Denote the classes of admissible controls for the firm-$i$ by $\mathcal {U}_i$, $i\in\{0,1\}$.

Now we consider a pair of optimal control problems on $[0,T]$ as follows.
\begin{equation}\label{9.1}\tag{{\bf{P1}}}
\begin{split}
 \text{Leader's\ problem:}&\max_{u_{0}\in\mathcal{U}_0} J_{0}(u_{0})=\max_{u_{0}\in\mathcal{U}_0}\big\{u_{0}(t)(a-b(u_{1}(t)+u_{0}(t))-c_{0})\big\},\\
 \text{Follower's\ problem:}& \max_{u_{1}\in\mathcal{U}_1} J_{1}(u_{1})=\max_{u_{1}\in\mathcal{U}_1}\big\{u_{1}(t)(a-b(u_{1}(t)+u_{0}(t))-c_{1})\big\}.
 \end{split}
\end{equation}

\noindent Denote \begin{itemize}
  \item the Stackelberg equilibrium by $(u^*_0,u^*_1)$, the defection strategy by $(\hat{u}_0,u^*_1)$.
  \item the payoff of firm-0 with the strategy $u_0^*$ in the $n^{th}$ game by $J^*_{0}(n)$.
  \item the payoff of firm-0 with the defection strategy $\hat{u}_0$ in the $n^{th}$ game by $\hat{J}_{0}(n)$.
  \item the extra payoff of $\hat{u}_0$ than $u_0^*$ in the $n^{th}$ game by $\Delta J_{0}(n):=\hat{J}_{0}(n)-J^*_{0}(n)$.
  \item the actual payoff of firm-0 with the punishment by the third party in the $n^{th}$ game by $\tilde{J}_{0}(n)$.
\end{itemize}

In the next section, we find that the leader can get extra gain with less production after defection. It means that the time inconsistency of the Stackelberg equilibrium is originated from the structure of the above Stackelberg games itself. Hence we introduce the third party in discrete-time Stackelberg games with the following  punishment rules. If the leader defects with $\hat{u}_0$ from the $m^{th}$ game, then the third party will discount his payoff in the $(m+1)^{th}$ game, and $\tilde{J}_{0}(n)$ is given by
\begin{equation}\label{9-1}
\tilde{J}_{0}(n)=\rho(n)(-bu_{0}^2+(a-bu_{1}-c_{0})u_{0}),
\end{equation}
where $\rho$ is a ``discount'' factor, which is related to time $n$.

Furthermore, we choose $\rho$ in the linear form with respect to $\Delta J_{0}(n)$ as
\begin{equation}\label{9.0}
\rho(n)=\begin{cases}
1, &\text{if}\  1\leq n \leq m-1, \\
\rho(n-1)-k\Delta J_{0}(n-1),  &\text{if}\ m\leq n \leq N,
\end{cases}
\end{equation}
where $k\in\big(0,(\Delta J_{0}(1))^{-1}\big)$. 
 Then we obtain the following result.
\begin{Thm}\label{1.1}
In the discrete-time Stackelberg games \eqref{9.1}, two firms repeat the above duopoly games $N$ times. Given the discount factor as \eqref{9.0} by the third party, then there exists a constant $k>0$ satisfies
\begin{equation}\label{con1}
(1-k\Delta J_{0}(1))^{N-m+1}\geq1-\frac{8}{9}k\Delta J_{0}(1)(N-m+1),
\end{equation}
such that  the payoff after defection must be less than the payoff with the Stackelberg equilibrium strategy, that is, $\sum_{n=1}^N\tilde{J}_0(n)\leq NJ^*_0(1)$.
Thus the firm-0 will keep the Stackelberg equilibrium during the entire game.
\end{Thm}

{\bf{Model 2: Dynamic Stackelberg games model.}} Now we extend the discrete-time case to the dynamic model.
Assume that the game period is $[0,T]$, $T<\infty$ and $r>0$ is the rate of discount. The unit cost of the follower is given by
\begin{equation}\label{1.11}
  c_{1}(x_{1}(t))=\begin{cases}
\bar{c}_{1}-\gamma x_{1}(t), &\text{if}\  x_{1}(t)< {\gamma}^{-1}(\bar{c}_{1}-c_{0}), \\
c_{0},                        &\text{if}\ x_{1}(t)\geq {\gamma}^{-1}(\bar{c}_{1}-c_{0}),
\end{cases}
\end{equation}
where $\gamma>0$ is the parameter describing the productivity of knowledge capital. Here we assume that the knowledge of the firm-1 can not exceed the level of the firm-0, that is $\gamma\leq\bar{c}_{1}$, since $x_1(t)\in[0,1]$. There is a pair of optimal control problems
\begin{equation}\label{1.10}\tag{{\bf{P2}}}
\begin{split}
  &\max_{u_{0}\in\mathcal{U}_0} J_{0}(u_{0})=\max_{u_{0}\in\mathcal{U}_0}\int^T_0 e^{-rt}\big\{u_{0}(t)(a-b(u_{1}(t)+u_{0}(t))-c_{0})\big\}dx,\\
  &\max_{u_{1}\in\mathcal{U}_1} J_{1}(u_{1})=\max_{u_{1}\in\mathcal{U}_1} \int^T_0 e^{-rt}\big\{u_{1}(t)(a-b(u_{1}(t)+u_{0}(t))-c_{1})\big\}dx.
 \end{split}
\end{equation}

Now we introduce the third party into dynamic Stackelberg games, the ``discount'' factor denoted as $\rho(t)$, where $t$ is the duration of defection. Assume that the leader defected at $t_{0}\in[0,T]$ by $\hat{u}_{0}$, then his payoff is given by
\begin{align}\label{3.4}
 \tilde{J}_0&=\int^{t_{0}}_0 e^{-rt}\big\{u^*_{0}(t)(a-b(u^*_{0}(t)+u^*_{1}(t))-c_0)\big\}dt\\\nonumber
&\quad+\int^T_{t_{0}}e^{-rt}\rho(t)\big\{\hat{u}_{0}(t)(a-b(\hat{u}_{0}(t)+u^*_{1}(t))-c_0)\big\}dt.
\end{align}

Here the discount $\rho(t)$ satisfies
 \begin{equation}\label{rho}
\rho'(t)\leq0,\  \rho''(t)\geq0,\ \text{and}\ \rho(0)=1,\  \rho(t)\in[0,1].
 \end{equation}
\begin{Rem}\label{Rem1.1}
We give the economic explanation of the conditions \eqref{rho}.
  Before the defection, the leader must not be punished, that is $\rho(0)=1$. If the leader honours his commitment during the entire game, then the discount is equal to 1. Otherwise, his actual payoff will be discounted by the third party, hence $\rho\in[0,1]$. On the one hand, the longer time of defection lasts, the larger discount factor becomes, that is, $\rho'(\tau)\leq0$. On the other hand, $-\rho(\tau)$ can be regarded as a special production function satisfies Inada conditions, such that $\rho''(\tau)\geq0$. Naturally, $\rho(\tau)\geq0$, indeed, the leader can stop production when $\rho$ is negative.
\end{Rem}
Further we choose the discount $\rho$ in the exponential form with respect to $t$ as
\begin{equation}\label{r2}
\rho(t)=\begin{cases}
1, &\text{if}\  0<t\leq t_0, \\
e^{-k(t-t_0)},  &\text{if}\ t_0<t\leq T,
\end{cases}
\end{equation}
where $k>0$ and $t_0$ is the defection time. Then we obtain the following result.

\begin{Thm}\label{1.2}
In the dynamic Stackelberg games \eqref{1.10},  given the discount factor as \eqref{r2}, then there exists a constant $k>0$ satisfies
\begin{align}\label{equ110}
\begin{split}
&\quad\frac{A_{1}(e^{(s_{1}-r)T}-1)}{s_{1}-r}+\frac{A_{2}(e^{(s_{2}-r)T}-1)}{s_{2}-r}
+\frac{A_{3}(e^{(2s_{1}-r)T}-1)}{2s_{1}-r}\\
&+\frac{A_{4}(e^{(2s_{2}-r)T}-1)}{2s_{2}-r}+\frac{A_{5}(e^{(s_{1}+s_{2}-r)T}-1)}{s_{1}+s_{2}-r}
-\frac{A_{6}(e^{-rT}-1)}{r}\\
&+\frac{A_{7}(e^{(s_{1}-k-r)T}-1)}{s_{1}-k-r}+\frac{A_{8}(e^{(s_{2}-k-r)T}-1)}{s_{2}-k-r}+\frac{A_{9}(e^{(2s_{1}-k-r)T}-1)}{2s_{1}-k-r}\\
&+\frac{A_{10}(e^{(2s_{2}-k-r)T}-1)}{2s_{2}-k-r}+\frac{A_{11}(e^{(s_{1}+s_{2}-k-r)T}-1)}{s_{1}+s_{2}-k-r}-\frac{A_{12}(e^{-(k+r)T}-1)}{k+r}\geq0,
\end{split}
\end{align}
such that  $\tilde{J}_0<J^*_0$.
Thus the firm-0 will keep Stackelberg equilibrium during the entire game. Here $s_1, s_2,A_n,\ n\in\{1,2,\cdots,12\}$ are constants given in the appendix.
\end{Thm}
{\bf{Model 3: Mean field Stackelberg games model.}} In the market, there are a large number of minor firms-$j$, $j\in\{1,\cdots,N\}$. The major firm-0 as a leader has a significant impact on the others who are the followers.
 Although each minor firm could not affect the major firm, the average state of the minor firms, denoted by $\bar{x}(t)$,
 has a great impact on the major one. Denote the knowledge stocks of firm-$i$, $i\in\{0,1,\dots,N+1\}$ by $x_i$ satisfying
\begin{equation}\label{equ-51}
dx_{0}=\left(A_{0}x_{0}+B_{0}u_{0}+C_{0}\bar{x}\right)dt+\sqrt{x_0(1-x_0)}dW^{0}_{t},\quad x_0(0)=x^0_0,
\end{equation}
\begin{equation}\label{equ-52}
  dx_{j}=\left(Ax_{j}+Bu_{j}+C\bar{x}+Dx_{0}\right)dt+\sqrt{x_j(1-x_j)}dW^{j}_{t},\quad x_j(0)=x^0_j,
\end{equation}
where $u_i$ is the output of the firm-$i$, Brownian motions $\{W_{t}^{0}\}_{t\in[0,T]}$ and $\{W_{t}^{j}\}_{t\in[0,T]}$ are independent, standard one dimensional Brownian motions. Others parameters are constants. 
The fact that the expectation of stochastic integral with respect to
martingale is zero will help us to solve the problems. Then there is a pair of optimal control problems as follows.
\begin{equation}\label{equ-53}\tag{{\bf{P3}}}
\begin{split}
 & \max_{u_{0}\in\mathcal{U}_0} J_{0}(u_{0})=\max_{u_{0}\in\mathcal{U}_0}\mathbf{E}\bigg[\int^{T}_{0}e^{-rt}\left\{-a_{0}u^{2}_{0}
 +(x_{0}-l_{0}\bar{x}+b_{0})^{2}\right\}dt\bigg],\\
 &\max_{u_{j}\in\mathcal{U}_j} J_{j}(u_{j})=\max_{u_{j}\in\mathcal{U}_j}\mathbf{E}\bigg[\int^{T}_{0}e^{-rt}\big\{-au^{2}_{j}
 +\left(x_{j}-l\bar{x}
 +b_{j}\right)^{2}-2\sigma u_{0}u_{j}\big\}dt\bigg],
\end{split}
\end{equation}
where $l_0,l,b_0,\sigma$ are some positive parameters.
 By the Pontryagin maximum principle and variational method, we obtain Stackelberg equilibrium $(u^*_0,u^*_1)$
given in Lemma \ref{5-2} and defection group $(\hat{u}_0,u^*_1)$ given in Lemma \ref{5-3}.

Now we introduce the third party  with the discount factor $\rho(t)$ given in \eqref{r2}. Assume that the leader defects at $t=0$ by $\hat{u}_0$, then his actual payoff is given by
$$\tilde{J}_0=\mathbf{E}\bigg[\int^{T}_{0}e^{-\tilde{r}t}\left\{-a_{0}\hat{u}_{0}^{2}
+\left(\hat{x}_{0}-l_{0}\bar{x}+b_{0}\right)^{2}\right\}dt\bigg],$$
where $\tilde{r}=r+k$. Then we obtain the following result.
\begin{Thm}\label{1.3}
In the mean field Stackelberg games \eqref{equ-53},  given the discount
factor as \eqref{r2},  if
\begin{equation}\label{5067}
 O(\hat{x}_{0}(t))<O(e^{\frac{\tilde{r}}{2}t}),
\end{equation}
 then there exists a constant $k>0$ such that
$\tilde{J}_0<J^*_0.$
Thus the firm-0 will keep Stackelberg equilibrium.
Here $\hat{x}_0$ is the leader's state after defection, and $O(\cdot)$ represents the order of the growth of $\tilde{r}$.
\end{Thm}
\begin{Rem}\label{192}
The condition \eqref{5067} means that the rate of knowledge growth with
the interest rate $r$ should be slower than exponential growth, which is a natural condition.
\end{Rem}

 The main contribution of this paper is to introduce the third party to prevent defection caused by the time inconsistency
 in three different Stackelberg games, which are the discrete-time games, the dynamic games, and the MFG.

 Firstly, the role of the third party is to promote the realization of Stackelberg equilibrium. Once the leader defects,
 the third party will discount his payoff as punishment. The idea of ``discount'' punishment is inspired by the trigger strategy.
 Different from the trigger strategies in Nash games, the third party imposes the punishment rather than other players in Stackelberg games.  It makes punishment immediate without a margin of payoff by delay,
 and the effectiveness of the discount will not even depend on the strategy of charging or sacrificing other followers.

 Secondly, the introduction of the third party is very maneuverable in practice. Compared with precommitment and equilibrium strategies,
  Stackelberg equilibrium with the third party may enable both sides of the game to gain more benefits. Therefore, even if the leader needs to pay  a margin and followers need to pay a certain fee, they will be willing to introduce the third party to promote Stackelberg equilibrium.
  By the game rules we made, we know that the third party can suppress defection by raising the leader's interest rate to a certain level.

 Thirdly, our conclusions are consistent, although different forms of the discount factor are introduced into different models.
 It has a certain theoretical significance.  Unlike the first two games, the explicit expression of optimal control
 cannot be given in the mean field Stackelberg games. Here we give the  appreciated penalties by the third party using the order of the growth of the  interest rate. This condition is reasonable given in Remark \ref{192}.

The rest of the paper is organized as follows. In Section 2, we introduce the third party in discrete-time games
and obtain the result given in Theorem \ref{1.1}, which makes the Stackelberg equilibrium credible. In Section 3, we obtain the Stackelberg equilibrium for the dynamic games by Pontryagin maximum principle. Then we choose an appropriate discount and prove Theorem \ref{1.2}.
 Section 4 focuses on mean field Stackelberg games. Under the mean field approximation, we obtain the Stackelberg equilibrium and defection
  strategy. Then we obtain the appreciated penalties by the third party given in Theorem \ref{1.3}.

\section{DISCRETE-TIME STACKELBERG GAMES}
In this section, we discuss the model of discrete-time Stackelberg games. There is a pair of optimal control problems given in \eqref{9.1} satisfying the state equation \eqref{9.2}. First, we get the Stackelberg equilibrium and the defected strategy of the leader, see Lemma \ref{001}. To make the leader keep his promise, we introduce the third party, and set appropriate mandatory punishment rules, see Theorem \ref{1.1}.
\subsection{Stackelberg equilibrium}
\begin{Lem}\label{001}
In the one-shot Stackelberg games \eqref{9.1}, there is the Stackelberg equilibrium
\begin{equation}\label{1.5}
(u_{0}^*, u_{1}^*)=\bigg(\frac{a+c_{1}-2c_{0}}{2b},\frac{a+2c_{0}-3c_{1}}{4b}\bigg),
\end{equation}
and the payoff of the leader is given by
\begin{equation}\label{1.6}
J_{0}^*(u_{0}^*,u_{1}^*)=\frac{(a+c_{1}-2c_{0})^2}{8b}.
\end{equation}
Furthermore, if the defection is allowed, then  there is an optimal defection strategy
\begin{equation*}
  \hat{u}_{0}=\frac{3(a+c_{1}-2c_{0})}{8b},
\end{equation*}
and the payoff of the leader is given by
\begin{equation}\label{1.8}
  \hat{J}_{0}(1)=\hat{J}_{0}(\hat{u}_{0},u^{*}_{1})=\frac{9}{64b}(a+c_{1}-2c_{0})^2.
\end{equation}
\end{Lem}
\begin{proof}
First we assume that the leader makes a commitment to carry out the output strategy $u_{0}$ over a number of periods, and  the follower believes this commitment and makes the optimal output strategy as $u_{1}$ under this belief. Based on the leader's strategy $u_{0}$, the rational follower must adopt  the best reply function as follows
\begin{equation}\label{333}
u_{1}^*=R_{1}(u_{0})=\frac{a-bu_{0}-c_{1}}{2b}.
\end{equation}
After observing the output of the follower $u_{1}^*$, the leader solves his optimal output $u^*_0$ by maximizing his payoff
$$\max_{u_0\in\mathcal {U}_0} J_{0}(u_{0},R_{0}(u_{1}))=\max_{u_0\in\mathcal {U}_0}\left\{-\frac{b}{2} u_{0}^2+\frac{1}{2} (a+c_{1}-2c_{0})u_{0}\right\}.$$
It is easy to see that the objective function is the maximum value of the quadratic function of $u_{0}$, then we obtain
$$u_{0}^* =\frac{a+c_{1}-2c_{0}}{2b},\quad u_{1}^* =R_{1}(u_{0}^*)=\frac{a+2c_{0}-3c_{1}}{4b}.$$
If the leader keeps his promise, then we get Stackelberg equilibrium \eqref{1.5}.
Clearly, his payoff is given by \eqref{1.6}.

If the defection is allowed, the follower believes the leader's commitment $u_{0}^*$  and chooses his output. Then the leader deviates from its commitment without the follower's knowledge when the time comes for both firms to produce simultaneously. Further the rational firm-0 must adopt the best reply function
 $$\hat{u}_{0}=R_{0}(u_{1}^*)=\frac{3(a+c_{1}-2c_{0})}{8b} < u_{0}^*,$$
and his payoff is given by \eqref{1.8}. It is easy to see that $\hat{J}_{0}(\hat{u}_{0},u^{*}_{1})>J_{0}^*(u_{0}^*,u_{1}^*)$.
\end{proof}

Note that Lemma \ref{001} implies that if the leader defects, he will earn more payoff with less output. The time inconsistency leads the Stackelberg equilibrium to be incredible. In the following, we introduce the third party in the game.
\subsection{Discount by the third party}
 By comparison, defection can get extra gain 
\begin{eqnarray}\label{equ-11}
\Delta J_{0}(1)= \hat{J}_{0}(\hat{u}_{0},u_{1}^*)
-J^*_{0}(u_{0}^*,u_{1}^*)=\frac{1}{64b}(a+c_{1}-2c_{0})^2>0.
\end{eqnarray}
In the one-shot game, the leader must defect based on the rational assumption of players. Therefore the follower will not believe any commitment made by the leader in this way. In the repeated game, we introduce the third party with the discount factor $\rho(n)$ given in \eqref{9.0}. We emphasize that the third party provides a way to achieve Stackelberg equilibrium with the time inconsistency for finite repeated games. Now we discuss the appreciated penalties of the discount factor.


\begin{proof}[{\bf{Proof of Theorem \ref{1.1}}}]
Firstly, it is clearly that as long as the leader defects within the game, he must keep defecting from a certain game to the end. Otherwise, his actual payoff will be continuously discounted by the third party in advance. Hence we assume that the leader keeps defecting from the $m^{th}$ game and the discount factor satisfies
 $$\rho(1)=\rho(2)=\cdots=\rho(m)=1.$$
 When $m=N$, it implies that in the last time period, the leader wants to defect. Hence the leader  needs to pay the deposit $\Delta J_{0}(1)$ given in \eqref{equ-11} to the third party before the games.
For any integer $1\leq m<n\leq N$, let $\Delta \tilde{J}_{0}(n):=\tilde{J}_{0}(n)-J_{0}(n)$, where $\tilde{J}_{0}(n)$ defined in \eqref{9-1}.
 Then we claim that
$$\rho(n)=(1-k\Delta J_{0}(1))^{n-m},\ \text{and}\ \Delta \tilde{J}_{0}(n)=
\Delta J_{0}(1)(1-k\Delta J_{0}(1))^{n-m}.$$
We prove the claim by induction.  Fixed $m=1$, when $n=2$, we have
\begin{align*}
\rho(2)&=1-k\Delta J_{0}(1),\\
\Delta \tilde{J}_{0}(2)&=
\rho(2)(\hat{J}_{0}(1)-J_{0}^*(1))=(1-k\Delta J_{0}(1))\Delta J_{0}(1).
\end{align*}
The claim holds for $n=2$. Using the result for $n$, we have
\begin{align*}
\rho(n+1)&=\rho(n)-k\Delta\tilde{J}_{0}(n)
=(1-k\Delta J_{0}(1))^{n-m+1},\\
\Delta \tilde{J}_{0}(n+1)&=\rho(n)\Delta J_{0}(1)=(1-k\Delta J_{0}(1))^{n-m}\Delta J_{0}(1).
\end{align*}
Similarly, for any $1\leq m\leq N$, the claim is valid.

It follows from \eqref{1.6}, \eqref{1.8} and \eqref{equ-11} that $\hat{J}_{0}(1)=9\Delta J_{0}(1)$, and $ J_{0}^*(1)=8\Delta J_{0}(1).$
Then the total payoff of leader is given by
$$
 \sum_{n=1}^N\tilde{J}_0(n)=(m-1)J_{0}^*(1)+\sum_{n=m}^{N}\rho(n)\hat{J}_{0}(1).
$$
It is worth noting that, for the firm-0, if his cumulative payoff of $N$ times after defection is less than that of non-defection, that is, $\sum^N_{n=m+1}\tilde{J}_0(n)\leq(N-m)J^*_0(1)$, thus the firm-0 must honour his commitment.
Hence we have
\begin{equation}\label{1.9}
 9\Delta J_0(1)\sum_{n=m}^{N}(1-k\Delta J_{0}(1))^{n-m} \leq 8(N-m+1)\Delta J_0(1),
\end{equation}
and the inequality \eqref{1.9} is converted to
$$(1-k\Delta J_{0}(1))^{N-m+1}\geq 1-\frac{8}{9}(N-m+1)(k\Delta J_{0}(1)).$$
 By numerical method, we solve the transcendental equation for the factor $k$ depending on $N$, and the result follows.
\end{proof}

\section{DYNAMIC STACKELBERG GAMES}
In this section, we discuss the model of dynamic Stackelberg games. There is a pair of optimal control problems given in \eqref{1.10}
satisfying the state equation \eqref{9.2} with \eqref{1.11}. First, we get the Stackelberg equilibrium and the defected strategy
of the leader, see Lemma \ref{002}. Then we introduce the third party and set
appropriate mandatory punishment rules, see Theorem \ref{1.2}.

Without loss of generality, we can set the unit cost of the leader $c_{0}\equiv0$. Indeed, it suffices otherwise to consider $$c_{1}(x_{1}(t))-c_{0}=\begin{cases}
\bar{c}_{1}-c_{0}-\gamma x_{1}(t), &\text{if}\ x_{1}<\frac{\bar{c}_{1}-c_{0}}{\gamma }, \\
0, &\text{if}\  x_{1}\geq\frac{\bar{c}_{1}-c_{0}}{\gamma }.
\end{cases}$$
Hence we just need to discuss the utility functions by supposing that $\tilde{a}:=a-c_{0}$.
The optimal control problems \eqref{1.10} is rephrased as follows. The gain of the firm-1 is to find a control $u_1\in\mathcal {U}_1$, which maximizes his payoff
$$\max_{u_{1}\in\mathcal{U}_1}\int^T_0 e^{-rt}u_{1}(t)\big(a-b(u_{1}(t)+u_{0}(t))+\gamma x_{1}(t)-\bar{c}_1\big)dt,$$
and subjects to \eqref{9.2}. Accordingly, the firm-0 needs to maximizes his payoff
$$\max_{u_{0}\in\mathcal{U}_0}\int^T_0 e^{-rt}\big(-bu^2_{0}(t)+(a-bu_1(t))u_0(t)\big)dt,$$
with the invariable knowledge stock $x_0(0)$. In fact, as a mature firm, his knowledge has reached the highest level, hence it is better to make payoff by selling the knowledge than by improving it expensively, so his knowledge stock satisfies \eqref{x0}.

Note that if $\gamma$ is identically vanishing, it implies that there is no reaction between the unit cost of minor firm and his knowledge stock. This situation corresponds a time period of the discrete-time differential game described in the  last section, which is a special case of the dynamic model.

\subsection{Stackelberg equilibrium}
In this subsection, similar to the discrete-time game, we first solve an open-loop Stackelberg equilibrium. Then if defection is allowed, we obtain the defection strategy.
\begin{Lem}\label{002}
In the dynamic Stackelberg games \eqref{9.1}, let $\Delta:=r^2+4\delta(r+\delta)-\frac{4\gamma}{b}(r+2\delta)$, if $\Delta>r^2$, then there is a Stackelberg equilibrium
\begin{equation}\label{3.1}
(u_{0}^*(t), u_{1}^*(t))=\left(\frac{a+\bar{c}_{1}-\gamma x^*_{1}(t)-\lambda^*_1(t)}{2b},
\frac{a-3\bar{c}_{1}+3\gamma x^*_{1}(t)+\lambda^*_1(t)}{4b}\right),
\end{equation}
satisfying
\begin{align}\label{209}
\begin{split}
\left(
  \begin{array}{cc}
    x^*_{1}(t) \\
    \lambda^*(t)\\
  \end{array}
\right)
=\frac{1}{\sqrt{\Delta}}\left(
                              \begin{array}{c}
\lambda^*(0)(e^{s_{2}t}-e^{s_{1}t})+\frac{\alpha_{2}}{4bs_{1}}(e^{s_{1}t}-1)-\frac{\alpha_{1}}{4bs_{2}}(e^{s_{2}t}-1)\\
\lambda^*(0)(q_{2}e^{s_{2}t}-q_{1}e^{s_{1}t})+\frac{q_{1}\alpha_{2}}{4bs_{1}}(e^{s_{1}t}-1)-\frac{q_{2}\alpha_{1}}{4bs_{2}}(e^{s_{2}t}-1)\\
                              \end{array}
                            \right),
\end{split}
\end{align}
where $s_{1}=\frac{r-\sqrt{\Delta}}{2}$, $s_{2}=\frac{r+\sqrt{\Delta}}{2}$, $q_{i}=4b(s_{i}+\delta-\frac{3\gamma}{4b})$,
 $\alpha_{i}=q_{i}(a-3\bar{c}_{1})-\gamma(a+\bar{c}_{1})$,
 \begin{equation}\label{7.20}
  \lambda^*(0)=\frac{q_{1}\alpha_{2}s_{2}(e^{s_{1}T}-1)-q_{2}\alpha_{1}s_{1}(e^{s_{2}T}-1)}
{4bs_{1}s_{2}(q_{1}e^{s_{1}T}-q_{2}e^{s_{2}T})}.
\end{equation}
The payoff of the leader is given by
\begin{equation}\label{3.3}
 J^*_{0}(u^*_{0})=\int^T_{0}e^{-rt}\frac{(a+\bar{c}_{1}-\gamma x^*_{1}(t))^2-({\lambda}^*(t))^2}{8b}dt.
\end{equation}
If the defection is allowed, then  there is an optimal defection strategy for the leader
\begin{eqnarray}\label{4.4}
\hat{u}_{0}(t)=\frac{3(a+3\bar{c}_{1}-\gamma x^*_{1}(t))-{\lambda}^*(t)}{8b},
\end{eqnarray}
and the payoff of the leader is given by
\begin{equation}\label{3.2}
\hat{J}_{0}(\hat{u}_{0})=\int^T_{0}e^{-rt}\frac{(3(a+\bar{c}_{1}-\gamma x^*_{1}(t))-{\lambda}^*(t))^2}{64b}dt.
\end{equation}
\end{Lem}
\begin{proof}
First, we obtain the open-loop Stackelberg equilibrium by using the Pontryagin maximum principle. The optimal strategy of the follower always takes  the best reply function
\begin{eqnarray}\label{equ-12}
u_{1}(t)=R_{1}(u_{0}(t))=\frac{a-b u_{0}(t)+\gamma x_{1}(t)-\bar{c}_{1}}{2b},
\end{eqnarray}
with $$\dot{x}_{1}(t)=\frac{a-bu_{0}(t)+\gamma x_{1}(t)-\bar{c}_{1}}{2b}-\delta x_{1}(t).$$
 Substitute the rival's strategy $u_1(t)$ by \eqref{equ-12}, the leader maximizes his payoff
$$
\max_{u_{0}\in\mathcal {U}_0}\int^T_0 e^{-rt}\left\{-\frac{b}{2}u_{0}^2(t)+\frac{1}{2}(a+\bar{c}_{1}-\gamma x_{1}(t))u_{0}(t)\right\}dt,$$
Using the Pontryagin maximum principle, the Hamiltonian $H$ is given by
$$H
=-\frac{b}{2}u_{0}^2(t)+\frac{1}{2}(a+\bar{c}_{1}-\gamma x_{1}(t)-{\lambda}(t))u_{0}(t)+{\lambda}(t)\left(\frac{a+\gamma x_{1}(t)-\bar{c}_{1}}{2b}-\delta x_{1}(t)\right),$$
where ${\lambda}(t)$ is the adjoint variable associated with the knowledge stock $x_1(t)$, and satisfies the following adjoint equation
$$\dot{\lambda}_1(t)
=(r+\delta-\frac{3\gamma}{4b}){\lambda}(t)-\frac{\gamma^2}{4b}x_{1}(t)
+\frac{\gamma(a+\bar{c}_{1})}{4b},\ \text{and}\ {\lambda}(T)=0.$$

For the leader, his output satisfies maximum condition
 $$\frac{\partial H}{\partial u_{0}}=-bu_{0}(t)+\frac{1}{2}\big(a+\bar{c}_{1}-\gamma x_{1}(t)-{\lambda}(t)\big)=0,$$
then we can get the optimal control
$$u^*_{0}(t)=\frac{a+\bar{c}_{1}-\gamma x_{1}(t)-{\lambda}(t)}{2b},$$
corresponding to the reply function of the follower
\begin{eqnarray}\label{equ-14}
u_{1}^*(t)=R_{1}(u^*_{0}(t))=\frac{a-3\bar{c}_{1}+3\gamma x_{1}(t)+{\lambda}(t)}{4b}.
\end{eqnarray}
Substituting $u_{1}^*(t)$ into state equation and adjoint equation respectively,
we get autonomous equations, which are expressed in the following matrix form
$$
\left(
  \begin{array}{c}
    \dot{x}_{1}(t) \\
    \dot{\lambda}_1(t)\\
  \end{array}
\right)
=\left(
   \begin{array}{cc}
     \frac{3\gamma}{4b}-\delta & \frac{1}{4b}\\
     -\frac{\gamma^2}{4b}     & r+\delta-\frac{3\gamma}{4b}\\
   \end{array}
 \right)
 \left(
   \begin{array}{c}
     x_{1}(t) \\
     {\lambda}(t) \\
   \end{array}
 \right)
 +\left(
    \begin{array}{c}
      \frac{a-3\bar{c}_{1}}{4b} \\
      \frac{\gamma(a+\bar{c}_{1})}{4b}\\
    \end{array}
  \right),
 $$
with boundary conditions $x_{1}(0)=x^0_{1}$ and ${\lambda}(T)=0.$
Since $\Delta>r^2$, we obtain eigenvalues $s_{1}<0$, $s_{2}>0$, thus there is a pair of solutions $(x^*_1(t),\lambda^*(t))$.
By the complicated calculations, we get the Stackelberg equilibrium $(u^*_{0}(t),u^*_{1}(t))$, and his payoff is given by \eqref{3.3}.

Next, we obtain a defection strategy group\ $(\hat{u}_{0}(t),u^*_{1}(t))$.  Fixed the follower's strategy in \eqref{equ-14}, the leader needs to maximizes his payoff
$$\max_{u_{0}\in\mathcal{U}_0}\int^T_{0}e^{-rt}u_{0}(t)\big(a-b(u_{0}(t)+u^*_{1}(t))\big)dt,$$
where $u^*_{1}(t)$ is fixed function of $t$ in Stackelberg equilibrium. Hence the optimal output is the best reply function $\hat{u}_{0}(t)$ satisfying \eqref{333}. Clearly, the corresponding to his payoff satisfies \eqref{3.2}.
\end{proof}
\begin{Rem}
These results are consistent with the results given in Lemma \ref{001} for discrete-time games, corresponding to parameters $\gamma=\lambda^*=c_0=0$.
\end{Rem}
\subsection{Discount by the third party}
In the dynamic game, we introduce the third party with the discount factor $\rho(\tau)$ given in \eqref{rho}.  Now we prove Theorem \ref{1.2}.
\begin{proof}[{\bf{Proof of Theorem \ref{1.2}}}]
Without loss of generality, we assume that the leader defects at $t_0=0$. Indeed, by the time transformation,
 let $t'=t-t_0$ and $T'=T-t_0$. Then the actual payoff of the leader is given by
 $$\tilde{J}_{0}=\int^{T'}_0e^{-rt}\rho(t)\hat{u}_{0}(t)(a-b(\hat{u}_{0}(t)+u^*_{1}(t)))dt.$$
 As a supplement, the initial values of the state variable and the adjoint variable also need to be modified accordingly.
 To make the leader adopt the Stackelberg equilibrium, his payoff after defection must be less than non-defection payoff. Thus we need to calculate the factor $k$ such that $$\tilde{J}_{0}
\leq\int^{T}_0e^{-rt}u^*_{0}(t)(a-b(u^*_{0}(t)+u^*_{1}(t)))dt.$$
It follows from \eqref{3.1} and \eqref{4.4} that
\begin{eqnarray}\label{equ-20}
\int^T_0e^{-rt}\big\{(8-9\rho(t))(a+\bar{c}_{1}-\gamma x^*_{1}(t))^2-(8+\rho(t))({\lambda}^*(t))^2\nonumber\\+6\rho(t)(a+\bar{c}_{1}-\gamma x^*_{1}(t)){\lambda}^*(t)\big\} dt\geq0.
\end{eqnarray}

Now we discuss the discount factor as the exponential function $\rho(\tau)=e^{-k\tau},$ where $\tau$ is the duration of the defection.
Because of $\rho(0)=1$ and $\rho'(\tau)<0$, we get $k>0$. Then the discount $\rho(\tau)\rightarrow0$ as $k\rightarrow+\infty$. Furthermore, we have $J_{0}(\hat{u}_{0})\rightarrow0$. Thus we can effectively restrain defection by improving the value of $k$. By the numerical calculation, we can get the low bound of $k$.

Recall $\tilde{r}=r+k$, the payoff of the leader is given by
$$\tilde{J}_{0}=\int^{T}_0e^{-\tilde{r}t}\hat{u}_{0}(t)(a-b(\hat{u}_{0}(t)+u^*_{1}(t)))dt,$$
then $\tilde{r}$ can be regarded as the interest rate of the firm-0. Once the leader defects, his interest rate $\tilde{r}$ is raised by the third party immediately. It's worth noting that the interest rate of the follower is invariant. Hence there is no effect on $u^*_1$ and $x^*_1$ given in \eqref{3.1} and \eqref{209}, where $\lambda^*_1(t)$ satisfies \eqref{7.20} with replacing $r$ by $\tilde{r}$. It implies his actual payoff $\tilde{J}_{0}$ goes to zero, when $\tilde{r}$ goes to infinity.
Hence by calculating
$$\tilde{J}_{0}\leq\int^{T'}_0e^{-rt}u^*_{0}(t)(a-b(u^*_{0}(t)+u^*_{1}(t)))dt,$$
 we obtain the reasonable condition \eqref{equ110}, and the proof is completed.
\end{proof}

\section{MEAN FIELD STACKELBERG GAMES}
In this section, we extend the model to the mean field Stackelberg games with one major firm and a large number minor firms
as in \cite{13}. There is a pair of optimal control problems given in \eqref{equ-53} satisfying SDEs \eqref{equ-51} and \eqref{equ-52}.
First, we get the Stackelberg equilibrium given in Lemma \ref{12}, and the defected strategy given in Lemma \ref{5-3}.
Then we introduce the third party and set appropriate mandatory punishment rules, see Theorem \ref{1.3}.
\subsection{The representative follower's problem }
For the model \eqref{equ-53}, we first solve the optimal control problem for the representative follower, which refers to \cite{14}. Assume that the leader announces his output is
 $u_{0} \in \mathcal {U}_{0}$, since minor firms are coupled through only the mean field term, each follower actually faces a separate
 stochastic optimal control problem as follows.
\begin{Lem}{\rm{(The optimal control of representative follower)}}\label{5-1}
Given the leader's output $u_{0}\in \mathcal {U}_{0}$, for the representative follower-$i$, $i\in\{1,\cdots,N\}$, there exists a unique optimal control $u_{i}\in\mathcal {U}_i$. Moreover, $(x_{i} , u_{i})$ is the corresponding optimal solution if and only if
\begin{equation}\label{equ-3}
u_{i}(t)=-\frac{Bp_{i}(t)+\sigma u_{{0}}(t)}{a},
\end{equation}
where
\begin{displaymath}
\left\{
\begin{array}{l}\label{5.6}
dx_{i}=\left(Ax_{i}(t)-\frac{B^{2}}{a}p_{i}(t)-\frac{B\sigma}{a}u_{0}(t)+C\bar{x}(t)+Dx_{0}(t)\right)dt+\sqrt{x_0(1-x_0)}dW^i_{t},\\
d p_{i}=e^{-rt}\left(-Ap_{i}(t)+(x_{i}(t)-l\bar{x}(t)+b)\right)dt+m(t)dW^{i}_{t}+n(t)dW^{0}_{t},
\end{array}
\right.
\end{displaymath}
with the boundary conditions $x_{i}(0)=x_{i}^{0}$, and $p_{i}(T)=0$.
\end{Lem}
\begin{proof}
The existence and uniqueness of the optimal control follows from convexity and continuity of the gain function \eqref{equ-53} with respect to $u_i$, as well as its quadratic nature, see \cite{16} and \cite{17}.
Now we solve the optimal control given in \eqref{equ-52} by variational method. Here, we omit the variable $t$ for short. First we consider a perturbation of the control $u_{i}+\theta\tilde{u}_{i}$, corresponding the original state $x_{i}$ becomes $x_{i}+\theta\tilde{x}_{i}$ with $\tilde{x}_{i}(0)=0$ satisfying
$$d\tilde{x}_{i}=e^{-rt}(A\tilde{x}_{i}+B\tilde{u}_{i})dt.$$
The corresponding payoff satisfies
\begin{align*}
J_{i}(u_{i}+\theta\tilde{u}_{i})
&=\mathbf{E}\bigg[\int^{T}_{0}e^{-rt}\big\{-au_{i}^{2}-2\theta au_{i}\tilde{u}_{i}-\theta^{2}a\tilde{u}_{i}^{2}
+(x_{i}-l\bar{x}+b)^{2}\\
&\quad +2\theta\tilde{x}_{i}(x_{i}-l\bar{x}+b)+\theta^{2}\tilde{x}_{i}^{2}
-2eu_{0}u_{i}-2\theta \sigma u_{0}\tilde{u}_{i}\big\}dt\bigg].
\end{align*}
Then the optimality of $\tilde{u}_{i}$ would satisfy the following Euler condition
\begin{align}\label{equ-4}
0={\frac{\partial}{\partial\theta}} J_{i}(u_{i}+\theta\tilde{u}_{i})\bigg{|}_{\theta=0}=2\mathbf{E}\bigg[\int^{T}_{0}e^{-rt}\{-au_{i}\tilde{u}_{i}
+\tilde{x}_{i}(x_{i}-l\bar{x}+b)-\sigma u_{0}\tilde{u}_{i}\}dt\bigg].
\end{align}
Using the It\^{o} formula, we have the inner product
$$
d\left<\tilde{x}_{i},p_{i}\right>
=e^{-rt}\{B\tilde{u}_{i}p_{i}+\tilde{x}_{i}(x_{i}-l\bar{x}+b)\}dt
+m(t)\tilde{x}_{i}dW^i_{t}+n(t)\tilde{x}_{i}dW^{0}_{t}.
$$
Let
$$I_1:=\int_0^T m(t)\tilde{x}_{i}dW^i_{t},\ \text{and}\ I_2:=\int_0^T n(t)\tilde{x}_{i}dW^{0}_{t},$$
then $I_j,\ j\in\{1,2\}$ is a stochastic integral with respect to Brownian motion, and its expectation is zero. Taking integration and
expectation on both sides, it follows from \eqref{equ-4} that
$$0=\mathbf{E}\bigg[\int^{T}_{0}d\left<\tilde{x}_{i},p_{i}\right>\bigg]
=\mathbf{E}\bigg[\int^{T}_{0}e^{-rt}\tilde{u}_{i}\big(Bp_{i}+au_{i}+\sigma u_{0}\big)dt\bigg].$$
 As $\tilde{u}_{i}$ is arbitrary, the result follows.
\end{proof}

Further, by the transformation $p_{i}(t)=F_{i}(t)x_{i}(t)+f_{i}(t)$, we can obtain an equivalent state feedback representation of the optimal control
$$-u_{i}(t)=\frac{B}{a}F_{i}(t)x_{i}(t)+\frac{B}{a}f_{i}(t)+\frac{\sigma}{a}u_{0}(t).$$
And we can show that $F_{i}(t)$ satisfies the Riccati differential equation
$$\frac{dF_{i}(t)}{dt}+(1+e^{-rt})AF_{i}(t)-\frac{B^{2}}{a}F_{i}^{2}(t)-e^{-rt}=0,$$
and $f_{i}(t)$ satisfies the backward stochastic differential equation (briefly, BSDE)
\begin{align*}
df_{i}(t)&=\left(\frac{B^{2}}{a}F_{i}(t)-e^{-rt}A\right)f_{i}(t)dt+F_{i}(t)
\left(\frac{B\sigma}{a}u_{0}(t)-C\bar{x}(t)-Dx_{0}(t)\right)dt\\
&\quad+e^{-rt}(b-l\bar{x}(t))dt+\left(m(t)-F_{i}(t)\sqrt{x_i(1-x_i)}\right)dW^{i}_{t}+n(t)dW^{0}_{t},
\end{align*}
with terminal conditions $F_{i}(T)=f(T)=0$. Due to the limitation of paper space, the specific calculation process is omitted here.

\subsection{The leader's problem }
Next we consider the leader's optimal control problem which includes additional constraints induced by the mean field process determined
 by Nash followers. The recent paper \cite{13} implies that for a large number of followers, the impact of the followers on the leader
 collapses to the approximated MFG. For any process $g$, we denote
 $$g^{N}(t):=\frac{1}{N}\sum^{N}_{i=1}g_{i}(t),\quad \bar{g}(t):=\lim_{N\rightarrow\infty}g^{N}(t).$$
For simplicity, we omit the variable $t$, then we have
\begin{equation*}
\left\{
\begin{array}{l}
d\bar{x}=\big((A+C)\bar{x}-\frac{B^{2}}{a}\bar{p}-\frac{B\sigma}{a}u_{0}+Dx_{0}\big)dt,\\
d\bar{p}=e^{-rt}\{-A\bar{p}+(1-l)\bar{x}+b\}dt,\\
d\bar{f}=\big(\big{(}\frac{B^{2}}{a}\bar{F}-e^{-rt}A\big{)}\bar{f}-\big{(}l e^{-rt}+C\bar{F}\big{)}\bar{x}
 +\frac{B\sigma}{a}u_{0}\bar{F}-Dx_{0}\bar{F}+be^{-rt}\big)dt,
\end{array}
\right.
\end{equation*}
with boundary conditions $\bar{x}(0)=\bar{x}^0$, $\bar{p}(T)=\bar{f}(T)=0$.
\begin{Lem}{\rm{(The optimal control of the leader)}}\label{5-2}
Given optimal strategies of followers in \eqref{equ-3}, for the leader, there exists a unique optimal control $u_{0}\in\mathcal {U}_0$. Moreover, $(x_{0} , u_{0})$ is the corresponding optimal solution if and only if
\begin{eqnarray}\label{equ-5}
u^{*}_{0}(t)=\frac{B_{0}}{a_{0}}p_{0}(t)-\frac{Be}{aa_{0}}\lambda(t),
\end{eqnarray}
where
\begin{displaymath}
\left\{
\begin{array}{l}\label{equ-8}
dp_{0}=e^{-rt}\{-A_{0}p_{0}(t)-D\lambda(t)-(x_{0}(t)-l_{0}\bar{x}(t)+b_{0})\}dt+M(t)dW^{0}_{t},\\
d\lambda=e^{-rt}\{l_{0}(x_{0}(t)-l_{0}\bar{x}(t)+b_{0})-(A+C)\lambda(t)-C_{0}p_{0}(t)-(1-l)\xi(t)\}dt+N(t)dW^{0}_{t},\\
d\xi=e^{-rt}\{A\xi(t)+\frac{B^{2}}{a}\lambda(t)\}dt+P(t)dW^{0}_{t},
\end{array}
\right.
\end{displaymath}
with terminal conditions $p_{0}(T)=\lambda(T)=\xi(T)=0$.
\end{Lem}
\begin{proof}
The existence and uniqueness of the optimal control from a similar discussion as in the proof of Lemma \ref{5-1} and \cite{11}. Similarly, we prove \eqref{equ-5} by variational method. First we consider a perturbation of the control $u_{0}+\theta\tilde{u}_{0}$. We regard
 $x_{0}$, $\bar{x}$ and $\bar{p}$ as three states, which satisfy variations of SDEs as follows.
\begin{displaymath}
\left\{
\begin{array}{l}
 d\tilde{x}_{0}=e^{-rt}\{A_{0}\tilde{x}_{0}(t)+B_{0}\tilde{u}_{0}(t)+C_{0}\tilde{\bar{x}}(t)\}dt,\\
 d\tilde{\bar{x}}=e^{-rt}\{(A+C)\tilde{\bar{x}}(t)-\frac{B\sigma}{a}\tilde{u}_{0}(t)+D\tilde{x}_{0}(t)
 -\frac{B^{2}}{a}\tilde{\bar{p}}(t)\}dt,\\
 d\tilde{\bar{p}}=e^{-rt}\{-A\tilde{\bar{p}}(t)+(1-l)\tilde{\bar{x}}(t)\}dt,
\end{array}
\right.
\end{displaymath}
with boundary conditions $\tilde{x}_{0}(0)=\tilde{\bar{x}}_{0}(0)=\tilde{\bar{p}}(T)=0.$
 Considering
 $$J_{0}(u_{0}+\theta\tilde{u}_{0})=\mathbf{E}\bigg[\int^{T}_{0}e^{-rt}\big\{-a_{0}(u_{0}+\theta\tilde{u}_{0})^{2}
+\big(\theta(\tilde{x}_{0}-l_{0}\tilde{\bar{x}})+(x_{0}-l_{0}\bar{x}+b_{0})\big)^{2}\big\}dt\bigg],$$
then the optimality of $\tilde{u}_{0}$ satisfies the following Euler condition
\begin{align}\label{equ-6}
0&={\frac{\partial}{\partial\theta}} J_{0}(u_{0}+\theta\tilde{u}_{0})\bigg|_{\theta=0}=2\mathbf{E}\bigg[\int^{T}_{0}e^{-rt}\{-a_{0}u_{0}\tilde{u}_{0}
+(\tilde{x}_{0}-l_{0}\tilde{\bar{x}})(x_{0}-l_{0}\bar{x}+b_{0})\}dt\bigg].
\end{align}

Using the It\^{o} formula, we have the inner product
\begin{align*}
d\left<\tilde{x}_{0},p_{0}\right>&=e^{-rt}\{B_{0}\tilde{u}_{0}p_{0}+C_{0}\tilde{\bar{x}}p_{0}-D\lambda\tilde{x}_{0}
-\tilde{x}_{0}(x_{0}-l_{0}\bar{x}+b_{0})\}dt+\tilde{x}_{0}M(t)dW^{0}_{t},\\
d\left<\tilde{\bar{p}},\xi\right>&=e^{-rt}\{(1-l)\xi\tilde{\bar{x}}+\frac{B^{2}}{a}\lambda\tilde{\bar{p}}\}dt+\tilde{\bar{p}}P(t)dW^{0}_{t},\\
d\left<\tilde{\bar{x}},\lambda\right> &= e^{-rt}\{D\lambda\tilde{x}_{0}-C_{0}p_{0}\tilde{\bar{x}}-\frac{B\sigma}{a}\lambda\tilde{u}_{0}-\frac{B^{2}}{a}\tilde{\bar{p}}\lambda
-(1-l)\tilde{\bar{x}}\xi\\
 &\quad+l_{0}\tilde{\bar{x}}(x_{0}-l_{0}\bar{x}+b_{0})\}dt+\tilde{\bar{x}}N(t)dW^{0}_{t}.
\end{align*}
Taking integration and expectation on both sides and together with boundary conditions, we obtain
\begin{align*}
&\mathbf{E}\bigg[\int^{T}_{0}e^{-rt}\big\{B_{0}p_{0}\tilde{u}_{0}+C_{0}p_{0}\tilde{\bar{x}}-D\lambda\tilde{x}_{0}
-\tilde{x}_{0}(x_{0}-l_{0}\bar{x}+b_{0})\big\}dt\bigg]=0,\\
&\mathbf{E}\bigg[\int^{T}_{0}e^{-rt}\big\{(1-l)\xi\tilde{\bar{x}}+\frac{B^{2}}{a}\lambda\tilde{\bar{p}}\big\}dt\bigg]=0,\\
&\mathbf{E}\bigg[\int^{T}_{0}e^{-rt}\big\{D\lambda\tilde{x}_{0}-C_{0}p_{0}\tilde{\bar{x}}-\frac{B^{2}}{a}\lambda\tilde{\bar{p}}
-(1-l)\xi\tilde{\bar{x}}+l_{0}\tilde{\bar{x}}(x_{0}-l_{0}\bar{x}+b_{0})
-\frac{B\sigma}{a}\lambda\tilde{u}_{0}\big\}dt\bigg]=0,
\end{align*}
since the expectation of the stochastic integral with respect to Brownian motion is zero.
Combining the equation \eqref{equ-6}, we can obtain
$$\mathbf{E}\left[\int^{T}_{0}e^{-rt}\tilde{u}_{0}\big(-a_{0}u_{0}-\frac{B\sigma}{a}\lambda+B_{0}p_{0}\big)dt\right]=0.$$
As $\tilde{u}_{0}$ is arbitrary, the result follows.
\end{proof}
It follows from \eqref{equ-5} and \eqref{equ-3} that
\begin{eqnarray}\label{equ-55}
u^{*}_{i}(t)=-\frac{1}{a}\left(Bp_{i}(t)+\frac{B_{0}\sigma}{a_{0}}p_{0}(t)-\frac{B\sigma^{2}}{aa_{0}}\lambda(t)\right),
\ i\in\{1,\cdots,N\} .
\end{eqnarray}
Hence we obtain Stackelberg equilibrium $(u^{*}_{0},u^{*}_{i})$ satisfying \eqref{equ-5} and \eqref{equ-55}.
\begin{Lem}\label{006}\label{12}
In the mean field Stackelberg games \eqref{equ-53}, there is a Stackelberg equilibrium
$$(u_0^*(t),u_1^*(t))=\left(\frac{B_{0}}{a_{0}}p_{0}(t)-\frac{B\sigma}{aa_{0}}\lambda(t),\
\frac{B\sigma^{2}}{a^2a_{0}}\lambda(t)-\frac{B}{a}p_{i}(t)-\frac{B_{0}\sigma}{aa_{0}}p_{0}(t)\right),$$
where $p_{i}(t),\ p_{0}(t),\ \lambda(t)$ satisfy Lemma \ref{5-1} and Lemma \ref{5-2}, respectively.
The corresponding state equations are given by
\begin{align*}
dx^{*}_{i}&=\left(Ax^*_{i}(t)-\frac{B^{2}}{a}p_{i}(t)-\frac{BB_{0}\sigma}{aa_{0}}p_{0}(t)
             +\frac{B^{2}\sigma^{2}}{a^{2}a_{0}}\lambda(t)+C\bar{x}^*(t)+Dx_{0}(t)\right)dt\\
             &\quad+\sqrt{x_i(1-x_i)}dW^{i}_{t},\\
d\bar{x}^{*}&=\left((A+C)\bar{x}^*(t)+Dx^*_{0}(t)-\frac{B^{2}}{a}\bar{p}(t)-\frac{BB_{0}\sigma}{aa_{0}}p_{0}(t)
             +\frac{B^{2}\sigma^{2}}{a^{2}a_{0}}\lambda(t)\right)dt.
\end{align*}
\end{Lem}
\subsection{Defection strategy of the leader}
Fixed followers' strategies $u^{*}_{i}(t),\ i\in\{1,\cdots,N\}$, now we solve the one-player control problem of the leader
$$\max_{u_0\in\mathscr{U}_0}J_{0}(u_{0})=\max_{u_0\in\mathscr{U}_0}\mathbf{E}\bigg[\int^{T}_{0}e^{-rt}
\left\{-a_{0}u^{2}_{0}(t)+(x_{0}(t)-l_{0}\bar{x}^{*}(t)+b_{0})^{2}\right\}dt\bigg],$$
 subjects to
$$dx_{0}=\left(A_{0}x_{0}(t)+B_{0}u_{0}(t)+C_{0}\bar{x}^*(t)\right)dt+\sqrt{x_0(1-x_0)}dW^{0}_{t},$$
where stochastic processes $p_{0}(t),\ \bar{p}(t),\ \lambda(t)$ and $\bar{x}^{*}(t)$ are given.
\begin{Lem}{\rm{(The defection strategy for the leader)}}\label{5-3}
Given followers' strategies $u_{i}^{*}$ in \eqref{equ-3}, if defection is allowed, for the major firm, there exists a unique optimal control $\hat{u}_{0}$. Moreover, $(\hat{x}_{0} , \hat{u}_{0})$ is the corresponding optimal solution if and only if
\begin{eqnarray}\label{7.11}
 \hat{u}_{0}(t)=\frac{B_{0}}{2a_{0}}\zeta(t),
\end{eqnarray}
where
\begin{displaymath}\label{equ-56}
\left\{
\begin{array}{l}
  d\zeta = \left((r-A_{0})\zeta(t)-2(\hat{x}_{0}(t)-l_{0}\bar{x}^*(t)+b_{0})\right)dt+Z_{t}dW^{0}_{t}, \nonumber\\
  d\hat{x}_{0} = (A_{0}\hat{x}_{0}(t)+\frac{B^{2}_{0}}{2a_{0}}\zeta(t)+C_{0}\bar{x}^*(t))dt+\sqrt{x_0(1-x_0)}dW^{0}_{t}, \nonumber\\
  d\bar{x}^{*}=((A+C)\bar{x}^*(t)+Dx^*_{0}(t)-\frac{B^{2}}{a}\bar{p}(t)-\frac{BB_{0}\sigma}{aa_{0}}p_{0}(t)
  +\frac{B^2e^2}{a^{2}a_{0}}\lambda(t))dt,\nonumber\\
  dx^*_{0}=(A_{0}x^*_{0}(t)+\frac{B_{0}^{2}}{a}p_{0}(t)-\frac{BB_{0}\sigma}{aa_{0}}\lambda(t)+C_{0}\bar{x}^{*}(t))dt
              +\sqrt{x_0(1-x_0)}dW^{0}_{t},\nonumber\\
  d\bar{p}=e^{-rt}\{-A\bar{p}(t)+(1-l)\bar{x}^*(t)+b\}dt,
\end{array}
\right.
\end{displaymath}
with boundary conditions $\zeta(T)=\bar{p}(T)=0,\ \hat{x}_0(0)=x^*_0(0)=x_0^0$, and $\bar{x}^*(0)=\bar{x}^0$.
\end{Lem}
\begin{proof}
Since the convexity of the objective function, we apply the standard stochastic maximum principle \cite{7},  the Hamiltonian for the leader is given by
$$H(x_{0},\zeta,u_{0})=-a_{0}u^{2}_{0}(t)+(x_{0}(t)-l_{0}\bar{x}^{*}(t)+b_{0})^{2}
+\zeta(t)\big(A_{0}x_{0}(t)+B_{0}u_{0}(t)+C_{0}\bar{x}^*(t)\big),$$
where $\zeta(t)$ is the adjoint process of $X_{t}^0$ satisfying
$$d\zeta=\left(r\zeta(t)-\partial_{x_{0}}H(x_{0},\zeta,u_{0})\right)dt+Z_{t}^{0}dW^{0}_{t},$$
with terminal condition $\zeta(T)=0.$
The process $Z^{0}_{t}$ is adapted and square integrable.
Since the first optimal condition
$$\frac{\partial H}{\partial u_{0}}=-2a_{0}u_{0}(t)+B_{0}\zeta(t)=0,$$
the optimal strategy of the leader is given by
$$\hat{u}_{0}(t)=\frac{B_{0}}{2a_{0}}\zeta(t).$$
Substituting $u_{0}$ into the Hamiltonian, we get the adjoint equation
\begin{equation}\label{7.4}
d\zeta=\left((r-A_{0})\zeta(t)-2(x_{0}(t)-l_{0}\bar{x}^{*}(t)+b_{0})\right)dt+Z_{t}^{0}dW^{0}_{t}.
\end{equation}
Combining with the state equation
$$dx_{0}=\left(A_{0}x_{0}(t)+\frac{B^{2}_{0}}{2a_{0}}\zeta(t)+C_{0}\bar{x}^{*}(t)\right)dt+\sqrt{x_0(1-x_0)}dW^{0}_{t},$$
 which is a couple of affine forward-backward SDE, the result follows.
\end{proof}
Assume that a solution $\zeta(t)$ is given by an affine function of $X^{0}_{t}$, that is,
$$\zeta(t)=Q(t) x_{0}(t)+q(t),$$
then we get
\begin{align*}
   d\zeta&=\bigg(\big(\frac{dQ(t)}{dt}+A_{0}Q(t)+\frac{B^{2}_{0}}{2a_{0}}Q^{2}(t)\big)x_{0}(t)
+\frac{B^{2}_{0}}{2a_{0}} q(t)+C_{0}\bar{x}^{*}(t)Q(t)\bigg)dt\\\nonumber
& \quad +\sqrt{x_0(1-x_0)}Q(t)dW^{0}_{t}+dq(t).
\end{align*}
Substituting into adjoint equation \eqref{7.4}, we have
$$d\zeta=\big(((r-A_{0})Q(t)-2)x_{0}(t)+(r-A_{0})q(t)+2(l_{0}\bar{x}^{*}(t)-b_{0})\big)dt+Z_{t}^{0}dW^{0}_{t}.$$

Considering the corresponding coefficients of $x_{0}$ are equal, we show that the function $Q(t)$ must satisfy the scalar Riccati equation
$$\frac{dQ}{dt}=(r-2A_{0})Q(t)-\frac{B^{2}_{0}}{2a_{0}}Q^{2}(t)-2,$$
with terminal condition $Q(T)=0$.
Similarly, the function $q(t)$ must satisfy BSDE
\begin{align*}
  dq & =\left(\big(r-A_{0}-\frac{B^{2}_{0}}{2a_{0}}Q(t)\big)q(t)+(2l_{0}-C_{0}Q(t))\bar{x}^{*}(t)-2b_{0}\right)dt \\
  &\quad +\left(Z_{t}^{0}-\sqrt{x_0(t)(1-x_0(t))}Q(t)\right)dW^{0}_{t},
\end{align*}
with terminal condition $q(T)=0$.

\subsection{Discount by the third party}
Similar to the dynamic game, we introduce the third party with the the discount factor $\rho$ given in \eqref{r2}.         
Now we discuss the reasonable
conditions of the discount factor.
\begin{proof}[{\bf{Proof of Theorem \ref{1.3}}}] Consider that the actual payoff of the leader is given by
\begin{equation}\label{7.6}
\tilde{J}_0=\mathbf{E}\bigg[\int^{T}_{0}e^{-\tilde{r}t}\left\{-a_{0}\hat{u}_{0}^{2}+(\hat{x}_{0}-l_{0}\bar{x}+b_{0})^{2}\right\}dt\bigg],
\end{equation}
where $\tilde{r}=r+k$, then the third party can punish the defection by raising the interest rate $\tilde{r}$ of the leader.
If
\begin{equation}\label{344}
-a_{0}\hat{u}_{0}^{2}+(\hat{x}_{0}-l_{0}\bar{x}+b_{0})^{2}<\infty,
\end{equation}
then there exists a factor $k>0$ large enough such that the payoff after defection is less than non-defection payoff, that is, $\tilde{J}_{0}<J^*_0(u_0^*)$.

Now we claim that the  \eqref{5067} is a sufficient condition for \eqref{344}.
Notably, the third party does not raise the interest rate of followers. Hence there is no effect on $u_1^*(t)$, $\bar{x}_1^*(t)$ and $x_0^*(t)$, except $\zeta(t)$ and $\hat{x}_0(t)$ satisfying
\begin{align*}
  d\zeta &= \big((\tilde{r}-A_{0})\zeta(t)-2(\hat{x}_{0}(t)-l_{0}\bar{x}^*(t)+b_{0})\big)dt+Z_tdW_t^0, \\
  d\hat{x}_{0} &= \big(A_{0}\hat{x}_{0}(t)+\frac{B^{2}_{0}}{2a_{0}}\zeta(t)+C_{0}\bar{x}^*(t)\big)dt+\sqrt{x_0(1-x_0)}dW^0_t.
\end{align*}
Different from the dynamic game, we describe the  appreciated penalties in terms of the order of interest rate $\tilde{r}$. Specifically, we just need
$$O(\hat{u}_{0}(t))<O(e^{\frac{\tilde{r}}{2}t}).$$
Combined the gain function \eqref{7.6} with \eqref{7.11}, it is clearly that
\begin{equation}\label{O}
O(\zeta(t))<O(e^{\frac{\tilde{r}}{2}t}).
\end{equation}
 In fact, since the expectation of stochastic integral term is zero, this implies that the diffusion terms are invalid. Thus we just consider the ordinary differential equation
\begin{equation}\label{7.12}
 d\zeta = \big((\tilde{r}-A_{0})\zeta(t)-2(\hat{x}_{0}(t)-l_{0}\bar{x}^*(t)+b_{0})\big)dt,
\end{equation}
and ignore the terms unrelated to $\tilde{r}$. Using the method of constant variation to \eqref{7.12} with the terminal condition $\zeta(T)=0$, we get
\begin{align*}
\zeta(t) 
=2e^{(\tilde{r}-A_{0})t}\int^{T}_{t}e^{-(\tilde{r}-A_{0})s}(\hat{x}_{0}-l_{0}\bar{x}^*+b_{0})ds.
\end{align*}
Similarly, we can also obtain
$$\hat{x}_{0}(t)=e^{A_{0}t}\int^{t}_{0}e^{-A_{0}s}\left(\frac{B^2_{0}}{2a_{0}}\zeta(s)+c_{0}\bar{x}^*(s)\right)ds+Ce^{A_{0}t}.$$
It is clearly that $$O(\zeta(t))=O(\hat{x}_{0}(t)).$$
However, the condition \eqref{O} implies that the growth rate of knowledge with the interest rate $\tilde{r}$ should be slower than exponential growth, which is a natural condition. This is the reason that we give the condition for the knowledge stock $\hat{x}_0$.
Hence we can obtain a low bound of $k$ depending on $T$ by numerical calculations. The result follows.

\end{proof}
\section{CONCLUDING REMARKS}
In this paper, we obtain the Stackelberg equilibrium  by using the Pontryagin maximum principle and variational method for three different Stackelberg games, which are the discrete-time games, the dynamic games, and the MFG models. The time inconsistency of the Stackelberg equilibrium is originated from the games' structure itself. Time inconsistency in dynamic decision-making is often observed in social systems and daily life. Here we introduce the third party, the duty of which is achieving the Stackelberg equilibrium by supervising the leader's implementation and imposing penalties for the defection with the discount factor. Then we obtain different sufficient conditions of the factors in different game models to make the Stackelberg equilibrium credible. This discussion is similar to an inverse game problem, except that the third party here is not all players. A similar role of the third party is the industry association, government department, or banking institution. Although the forms of the factor are different, these results are consistent. This provides a theoretical basis for the game rules.

There are some advantages to the third party intervention. Firstly, compared with the known precommitment and equilibrium strategies, Stackelberg equilibrium may enable to gain more benefits for all players. Thus introducing a third party is a good approach for leaders and followers to implement the Stackelberg equilibrium strategy, even if players have to pay a certain fee. It can be regarded as
a modified naive strategy.
Secondly, the idea of ``discount'' punishment is inspired by the trigger strategy.
For the trigger strategy, all followers will adopt the punishment strategies after a defection in the Nash game. In this paper, the third party will punish the defector by discounting his payoff in the Stackelberg game. It makes punishment be immediately without a margin of payoff by delay, and the effectiveness of the discount will not even depend on the strategy of charging or sacrificing other followers. Thirdly, the third party can prevent defection effectively by raising the interest rate of the leader. It has practical significance.
\section{APPENDIX}
In this section, let $\lambda(0):=\tilde{\lambda}$ , and provide some specific values for \eqref{equ110} as follows.
\begin{align*}
 A_{1}&= e^{(s_1-r)t}\bigg\{\frac{-4\tilde{\lambda}\alpha_{2}q_{1}}{{\Delta}bs_{1}}+\frac{\alpha^2_1q^2_1}{{\Delta}b^2s_{1}^2}
 -\frac{\alpha_1\alpha_2q_1q_2}{{\Delta}b^2s_1s_2}+\frac{16\gamma(a+\bar{c}_1)}{\sqrt{{\Delta}}}\left(\tilde{\lambda}-\frac{\alpha_2}{4bs_1}\right)\\
 &\quad+\frac{\gamma^2}{{\Delta}}
 \left(\tilde{\lambda}\frac{4\alpha_2}{bs_1}-\frac{\alpha_2^2}{b^2s_1^2}+\frac{\alpha_1\alpha_2}{b^2s_1s_2}\right)\bigg\}, \\
 A_{2}&= e^{(s_2-r)t}\bigg\{\frac{4\tilde{\lambda}\alpha_1q_{1}}{{\Delta}bs_{2}}+\frac{{\alpha}^2_1q^2_2}{{\Delta}b^2s_{2}^2}
 -\frac{\alpha_1\alpha_2q_1q_2}{{\Delta}b^2s_1s_2}-\frac{16\gamma(a+\bar{c}_1)}{\sqrt{{\Delta}}}\left(\tilde{\lambda}-\frac{\alpha_1}{4bs_2}\right)\\
 &\quad +\frac{\gamma^2}{{\Delta}}\left(\frac{\alpha_1\alpha_2}{b^2s_1s_2}-\frac{4\alpha_1\tilde{\lambda}}{bs_2}-\frac{\alpha_1^2}{b^2s_2^2}\right)\bigg\},\\
 A_{3}&=e^{(2s_1-r)t}\bigg\{\frac{-8\tilde{\lambda}^2q^2_{1}}{{\Delta}}-\frac{\alpha^2_1q^{2}_1}{2{\Delta} b^2 s_{1}^2}
 +\frac{8\gamma^{2}\tilde{\lambda}^2}{{\Delta}}+\frac{\gamma^{2}\alpha_2^2}{2{\Delta} b^2 s^2_1}\bigg\},\\
 A_{4}&= e^{(2s_2-r)t}\bigg\{\frac{-8\tilde{\lambda}^2q^2_{2}}{{\Delta}}-\frac{\alpha^2_1q^{2}_2}{2{\Delta} b^2 s_{2}^2}
 +\frac{8\gamma^{2}\tilde{\lambda}^2}{{\Delta}}+\frac{\gamma^{2}\alpha_1^2}{2{\Delta} b^2 s^2_2}\bigg\},\\
 A_{5}&= e^{(s_1+s_2-r)t}\bigg\{\frac{16\tilde{\lambda}^2q_{1}q_{2}}{{\Delta}}+\frac{\alpha_1\alpha_2q_1q_2}{{\Delta} b^2 s_1s_{2}}-\frac{\alpha_1\alpha_2\gamma^2}{{\Delta} b^2s_1s_2}
 -\frac{16\tilde{\lambda}^2\gamma^2}{{\Delta}}\bigg\},\\
  A_{6}&= e^{-rt}\bigg\{\frac{\alpha_1\alpha_2q_1q_2}{{\Delta}b^2s_{1}s_2}-\frac{4\tilde{\lambda}\alpha_{1}q_{2}}{{\Delta}bs_{2}}-\frac{4\tilde{\lambda}\alpha_2q_{1}}{{\Delta}bs_{1}}
 -\frac{\alpha^2_1q^2_1}{2{\Delta}b^2s_{1}^2}-\frac{\alpha^2_1q^2_2}{2{\Delta}b^2s_{2}^2}+8(a+\bar{c}_1)^2
 \\
 &\quad +\frac{4\gamma\alpha_2(a+\bar{c}_1)}{\sqrt{{\Delta}}bs_1}
 -\frac{4\gamma\alpha_1(a+\bar{c}_1)}{\sqrt{{\Delta}}bs_2}
 +\frac{4\gamma^2\alpha_2\tilde{\lambda}}{{\Delta} bs_1}+\frac{4\gamma^2\alpha_1\tilde{\lambda}}{{\Delta} bs_2}
 +\frac{\alpha^2_1\gamma^2}{2{\Delta}b^2s_1^2}\\
&\quad +\frac{\alpha^2_1\gamma^2}{2{\Delta}b^2s_2^2}
 -\frac{\alpha_1\alpha_2\gamma^2}{{\Delta}b^2s_1s_2}\bigg\},\\
 A_{7}&= e^{(s_1-r-k)t}\bigg\{
 \frac{\alpha^2_1q^2_1}{8{\Delta}b^2s_{1}^2}-\frac{\alpha_1\alpha_2q_1q_2}{8{\Delta}b^2s_{1}s_2} -\frac{\alpha_2q_1\tilde{\lambda}}{2{\Delta}bs_1}
 +\frac{18\gamma(a+\bar{c}_1)}{\sqrt{{\Delta}}}\left(\frac{\alpha_2}{4bs_1}-\tilde{\lambda}\right)\\
 &\quad +\frac{9\gamma^2}{{\Delta}}\left(\frac{\alpha^2_1}{8b^2s^2_1}-\frac{\alpha_1\alpha_2}{8b^2s_1s_2}-\frac{\alpha_2\tilde{\lambda}}{2bs_1}\right)
 +\frac{6q_1(a+\bar{c}_1)}{\sqrt{{\Delta}}}\left(\frac{\alpha_2}{4bs_1}-\tilde{\lambda}\right)\\
 &\quad +\frac{6\gamma}{{\Delta}}(q_1+1)\left(\frac{\alpha_2}{4bs_1}-\tilde{\lambda}\right)\left(\frac{\alpha_2}{4bs_1}-\frac{\alpha_1}{4bs_2}\right)\bigg\},\\
  A_{8}&= e^{(s_2-r-k)t}\bigg\{
 \frac{\alpha^2_1q^2_2}{8{\Delta}b^2s_{2}^2}-\frac{\alpha_1\alpha_2q_1q_2}{8{\Delta}b^2s_{1}s_2} +\frac{\alpha_1q_2\tilde{\lambda}}{2{\Delta}bs_2}
 +\frac{(18\gamma+6q_2)(a+\bar{c}_1)}{\sqrt{{\Delta}}}\left(\tilde{\lambda}-\frac{\alpha_1}{4bs_2}\right)\\
 &\quad +\frac{9\gamma^2}{{\Delta}}\left(\frac{\alpha^2_1}{8b^2s^2_2}+\frac{\alpha_1\tilde{\lambda}}{2bs_2}-\frac{\alpha_1\alpha_2}{8b^2s_1s_2}\right)
 +\frac{6\gamma}{{\Delta}}(q_2+1)\left(\tilde{\lambda}-\frac{\alpha_1}{4bs_2}\right)\left(\frac{\alpha_2}{4bs_1}-\frac{\alpha_1}{4bs_2}\right)
 \bigg\},\\
A_{9}&=e^{(2s_1-r-k)t}\bigg\{
 \frac{-\tilde{\lambda}^2q^2_1}{{\Delta}}-\frac{\alpha^2_1q^2_1}{16{\Delta}b^2s^2_{1}} -\frac{9\gamma^2}{{\Delta}}\left(\tilde{\lambda}^2+\frac{\alpha^2_1}{16b^2s^2_1}\right)-\frac{6\gamma q_1}{{\Delta}}\left(\frac{\alpha_2}{4bs_1}-\tilde{\lambda}^2\right)\bigg\},\\
A_{10}&=e^{(2s_2-r-k)t}\bigg\{
 \frac{-\tilde{\lambda}^2q^2_2}{{\Delta}}-\frac{\alpha^2_1q^2_2}{16{\Delta}b^2s^2_{2}} -\frac{9\gamma^2}{{\Delta}}\left(\tilde{\lambda}^2+\frac{\alpha^2_1}{16b^2s^2_2}\right)+\frac{6\gamma q_2}{{\Delta}}\left(\frac{\alpha_1}{4bs_2}-\tilde{\lambda}^2\right)\bigg\},\\
 A_{11}&=e^{(s_1+s_2-r-k)t}\bigg\{
 \frac{2\tilde{\lambda}^2q_1q_2}{{\Delta}}+\frac{\alpha_1\alpha_2q_1q_2}{8{\Delta}b^2s_1s_{2}} +\frac{9\gamma^2}{{\Delta}}\left(2\tilde{\lambda}^2+\frac{\alpha_1\alpha_2}{8b^2s_1s_2}\right)\\
 &\quad+\frac{6\gamma}{{\Delta}}(q_1+q_2)\left(\frac{\alpha_2}{4bs_1}-\tilde{\lambda}\right)\left(\frac{\alpha_1}{4bs_1}-\tilde{\lambda}\right)\bigg\},\\
 A_{12}&= e^{-(r+k)t}\bigg\{\frac{\alpha_1\alpha_2q_1q_2}{8{\Delta}b^2s_1s_2}-
 \frac{\tilde{\lambda}\alpha_1 q_2}{2{\Delta}bs_{2}}-\frac{\tilde{\lambda}\alpha_2 q_1}{2{\Delta}bs_{1}} -\frac{\alpha^2_1q_1^2}{16{\Delta}b^2s^2_1}-\frac{\alpha^2_1q_2^2}{16{\Delta}b^2s^2_2}-9(a+\bar{c}_{1})^2\\
&\quad-\frac{18\gamma(a+\bar{c}_1)}{\sqrt{{\Delta}}}
 \left(\frac{\alpha_2}{4bs_1}-\frac{\alpha_1}{4bs_2}\right)-\frac{6(a+\bar{c}_1)}{\sqrt{{\Delta}}}\left(\frac{\alpha_2q_1}{4bs_1}-\frac{\alpha_1q_2}{4bs_2}\right)\\
&\quad -\frac{9\gamma^2}{{\Delta}}
 \left(\frac{\alpha_2\tilde{\lambda}}{2bs_1}+\frac{\alpha_1\tilde{\lambda}}{2bs_2}+\frac{\alpha^2_1}{16b^2s^2_1}
 +\frac{\alpha^2_1}{16b^2s^2_2}-\frac{\alpha_1\alpha_2}{8b^2s_1s_2}\right)-\frac{6\gamma}{\sqrt{{\Delta}}}\left(\frac{\alpha_2}{4bs_1}-\frac{\alpha_1}{4bs_2}\right)^2\bigg\}.
 \end{align*}
\\*

\noindent{\bf{Acknowledgments.}}
The authors thank sincerely to the referees for their careful reading and valuable comments.
\\*

\noindent{\bf{Disclosure statement.}} The authors declare that they have no conflict of interest.
\\*


\addcontentsline{toc}{section}{References} 
\end{document}